\documentclass[12pt,a4paper]{extarticle}
\usepackage[top=3.0cm,bottom=3.0cm,left=2.5cm,right=2.5cm]{geometry}

\usepackage[T1]{fontenc}
\usepackage{amsmath, graphicx, color, amsthm, amssymb, centernot, mnsymbol}
\usepackage{calrsfs}

\newtheorem{claim}{Claim}
\newtheorem{theorem}{Theorem}
\newtheorem{lemma}[theorem]{Lemma}

\newtheorem{conjecture}{Conjecture}

\newtheorem{observation}{Observation}

\newcommand{\set}[1]{\ensuremath{\left\{#1 \right\}}}
\newcommand{\skap}[1]{\ensuremath{\overline{\varphi}_\kappa^{#1}}}
\newcommand{\slam}[1]{\ensuremath{{\varphi}_\lambda^{#1}}}
\newcommand{\kap}[1]{\ensuremath{\varphi_\kappa^{#1}}}
\newcommand{\cir}[2]{\ensuremath{\varphi_\zeta^{(#1, #2)}}}

\def\le{\leqslant}
\def\ge{\geqslant}

\def\paren#1{\left( #1 \right)}

\def\angle#1{\left\langle #1 \right\rangle}

\newenvironment{proofclaim}[1][]%
    {\noindent \emph{Proof.} {}{#1}{}}{$~$\hfill $~\blacklozenge$ \vspace{0.2cm}}

\title{On non-repetitive sequences of arithmetic progressions:\\the cases $k \in \{4,5,6,7,8\}$}
\author
{
	Borut Lu\v{z}ar\thanks{Faculty of Information Studies, Novo mesto, Slovenia.  
			E-Mail: \texttt{borut.luzar@gmail.com}} \footnotemark[4],\quad
	Martina Mockov\v{c}iakov\'{a}\thanks{NTIS, University of West Bohemia, Pilsen, Czech Republic.
			E-Mail: \texttt{mmockov@ntis.zcu.cz}}, \quad
 	Pascal Ochem\thanks{LIRMM, Universit\'e Montpellier 2, Montpellier, France. \newline
    	E-Mails: \texttt{\{pascal.ochem,alexandre.pinlou\}@lirmm.fr}}, \\    
 	Alexandre Pinlou\footnotemark[3],\quad
	Roman Sot\'{a}k\thanks{Institute of Mathematics, Faculty of Science, Pavol J. \v{S}af\'arik University, Ko\v sice, Slovakia. 
			E-Mail: \texttt{roman.sotak@upjs.sk}} 
}

\begin{document}
\maketitle

{\abstract
{
	A \textit{$d$-subsequence} of a sequence $\varphi = x_1\dots x_n$ is a subsequence $x_i x_{i+d} x_{i+2d} \dots$, for any positive integer $d$ and
	any $i$, $1 \le i \le n$. A \textit{$k$-Thue sequence} is a sequence in which every $d$-subsequence, for $1 \le d \le k$, 
	is non-repetitive, i.e. it contains no consecutive equal subsequences.
	In 2002, Grytczuk proposed a conjecture that for any $k$, $k+2$ symbols are enough to construct a $k$-Thue sequences of arbitrary lengths.
	So far, the conjecture has been confirmed for $k \in \set{1,2,3,5}$. 
	Here, we present two different proving techniques, and confirm it for all $k$, with $2 \le k \le 8$.	
}}

\section{Introduction}
%%%%%%%%%%%%%%%%%%%%%%%%%%%%%%%%%%%%%%%%%%%%%%%%%%%%%%%%%%%%%%%%%%%%%%%%%%%%%

A \textit{repetition} in a sequence $\varphi$ is a subsequence $\rho = x_1 \dots x_{2t}$ 
of consecutive terms of $\varphi$ such that $x_i = x_{t+i}$ for every $i=1,\dots,t$. 
The length of a repetition is hence always even and comprised of two identical \textit{repetition blocks},
$\rho_1 = x_1\dots x_t$ and $\rho_2 = x_{t+1}\dots x_{2t}$. 
A sequence is called \textit{non-repetitive} or \textit{Thue} if it does not contain any repetition. 
Surprisingly, as shown by Thue~\cite{Thu06} (see \cite{Ber94} for a translation), 
having three distinct symbols suffices to construct non-repetitive sequences of arbitrary lengths.
This result is a fundamental piece in the theory of combinatorics on words. After that, 
a number of other concepts related to repetitions has been presented (see e.g.~\cite{BerPer07} for more details). 

In this paper, we continue dealing with the following generalization.
A (possibly infinite) sequence $\varphi$ is \textit{$k$-Thue} (or \textit{non-repetitive up to mod $k$}) 
if every $d$-subsequence of $\varphi$ is Thue, for $1 \le d \le k$. 
By a \textit{$d$-subsequence} of $\varphi$ we mean an arithmetic subsequence $x_i x_{i+d} x_{i+2d} \dots$ of $\varphi$.
Consider a sequence 
$$
	a \ \underline{b} \ d \ \underline{c} \ b \ \underline{c},
$$
which is Thue, but not $2$-Thue, since the $2$-subsequence $b \ c \ c$ is not Thue. 
On the other hand, 
$$
	\underline{a} \ b \ c \ \underline{a} \ d \ b
$$ is $2$-Thue, but not $3$-Thue, due to the repetition in the $3$-subsequence $a \ a$.

This generalization was introduced by Currie and Simpson~\cite{CurSim02} and has been immediately followed 
by an intriguing conjecture due to Grytczuk~\cite{Gry02}.
\begin{conjecture}[Grytczuk, 2002]
	\label{conj:k+2}
	For any positive integer $k$, $k+2$ distinct symbols suffice to construct a $k$-Thue sequence of any length.
\end{conjecture}
It is easy to show that having only $k+1$ symbols there is a repetition in any sequence of length at least $2k+2$,
so the bound $k+2$ is tight. 

Since $1$-Thue sequences are simply Thue sequences, the above mentioned result establishes the conjecture for $k=1$.
The conjecture has also been confirmed for $k=2$ in~\cite{CurSim02} and independently in~\cite{KraLuzMocSot15},
for $k=3$ in~\cite{CurSim02}, and for $k=5$ in~\cite{CurMoo03}. 
Although it has been considered also for the case $k=4$ by Currie and Pierce~\cite{CurPie03} 
using an application of the fixing block method, it remains open for all the cases except $k \in \set{1,2,3,5}$.

Several upper bounds have been established, first being $e^{33}k$ due to Grytczuk~\cite{Gry02}, 
and then substantially improved to $2k + O(\sqrt{k})$ in~\cite{GryKozWit11}. 
Currently the best known upper bound is due to Kranjc et al.~\cite{KraLuzMocSot15}.
\begin{theorem}[Kranjc et al., 2015]
	\label{thm:2k}
	For any integer $k \ge 2$, $2k$ distinct symbols suffice to construct a $k$-Thue sequence of any length.
\end{theorem}
\noindent The proof of the above is constructive and provides $k$-Thue sequences of given lengths.

The aim of this paper is two-fold. The main contribution is answering Conjecture~\ref{conj:k+2} in affirmative
for several additional values of $k$.
\begin{theorem}
	\label{thm:45678}
	For any $k \in \set{4,5,6,7,8}$, $k+2$ distinct symbols suffice to construct a $k$-Thue sequence of any length.
\end{theorem}
Moreover, we present two different techniques of proving the above theorem.
In the former, described in Section~\ref{sec:pas}, we use exhaustive computer search to determine morphisms 
for each $k$, $k \in \set{4,5,6,7,8}$, from which we construct $k$-Thue sequences.
In the latter, described in Section~\ref{sec:cons}, we use concatenation of special blocks given by another morphism. 
The purpose of the latter one is to introduce its ability to deal with larger $k$'s,  
therefore we only prove the cases $k=4$ and $k=6$. 
We believe, in the future, it could be used for proving Conjecture~\ref{conj:k+2} for infinitely many values of $k$.

%%%%%%%%%%%%%%%%%%%%%%%%%%%%%%%%%%%%%%%%%%%%%%%%%%%%%%%%%%%%%%%%%%%%%%%%%%%%%
\section{Preliminaries}
\label{sec:prel}
%%%%%%%%%%%%%%%%%%%%%%%%%%%%%%%%%%%%%%%%%%%%%%%%%%%%%%%%%%%%%%%%%%%%%%%%%%%%%%

In this section, we introduce additional terminology and notation used in the paper.
Throughout the paper, $i$ and $t$ are used to determine positive integers, unless more details are given.

An \textit{$\mathbb{A}$-sequence} (or simply a \textit{sequence} when the alphabet is known from the context or not relevant) 
of length $t$ is an ordered tuple of $t$ symbols from some alphabet $\mathbb{A}$.
Let $\varphi = x_1 \dots x_t$ be a sequence. 
A subsequence of $\varphi$ of consecutive terms $x_i \dots x_{j}$,
for some $i,j$, $1 \le i \le j \le t$, is denoted by $\varphi(i,j)$.
A \textit{term} indicates an element of a sequence at a specified index.
A \textit{block} is a subsequence of consecutive terms of some sequence.
When we refer to a term as a term of a block, by its index we mean the index of a term in the block.
We denote the term at index $i$ in a sequence $\varphi$ (resp. a block $\beta$)
by $\varphi(i)$ (resp. $\beta(i)$).

A \textit{prefix} of a sequence $\varphi = x_1\dots x_r$ is a sequence $\pi = x_1\dots x_s$, for some integer $s \le r$.
A \textit{suf\mbox{}f\mbox{}ix} is defined analogously.
In a sequence $\varphi$ consider a pair of sequences $\pi$ and $\varepsilon$ such that $\pi \varepsilon$ 
is a subsequence of $\varphi$, $\pi$ has length at least $1$, and $\varepsilon$ is a prefix of $\pi \varepsilon$.
The \emph{exponent} of $\pi \varepsilon$ is 
$$
	\exp(\pi \varepsilon) = \tfrac{|\pi \varepsilon|}{|\pi|}.
$$
If a sequence has exponent $p$, we call it a \textit{$p$-repetition}. 
A sequence is \textit{$q^+$-free} if it contains no $p$-repetition such that $p>q$.
For sequences over $3$-letter alphabets, Dejean~\cite{Dej72} proved the following.
\begin{theorem}[Dejean, 1972]
	\label{thm:dej}
	Over $3$-letter alphabets there exist $\frac{7}{4}^+$-free sequences of arbitrary lengths.
\end{theorem}

A \textit{morphism} is a mapping $\mu$ which assigns to each symbol of an alphabet a sequence. 
Applied to a sequence $\varphi$, $\mu(\varphi)$ is the sequence obtained from $\varphi$ where every
symbol is replaced by its image according to $\mu$. We say that a morphism is \textit{$k$-uniform} if it maps 
every symbol from the domain to some sequence of length $k$.

Given a sequence $\varphi = \beta_1 \dots \beta_t$ comprised of blocks $\beta_i$, for $1 \le i \le t$,
the \textit{covering subsequence} $\hat{\sigma}$ of a subsequence $\sigma$ in $\varphi$ 
is the subsequence $\varphi(i,j)$, where $i$ is the index of the first term of the block containing the first term of $\sigma$, 
and $j$ is the index of the last term of the block containing the last term of $\sigma$.

An \textit{$i$-shift} of $\varphi$ is the sequence $\varphi^i = x_{i+1} \dots x_{\ell} x_{1} \dots x_{i}$, i.e. the sequence $\varphi$ 
with the subsequence of the first $i$ elements moved to the end.
Let $\varphi$ be a sequence of length $\ell$. 
We define the \textit{circular sequence} $\cir{\ell}{t}$ of order $\ell$ and length $\ell^2 \cdot t$ as
$$
	\cir{\ell}{t} = \underbrace{\varphi^0 \varphi^1 \dots \varphi^{\ell-1} \ \dots \ \varphi^0 \varphi^1 \dots \varphi^{\ell-1}}_{t}.
$$
We call each subsequence $\varphi^i$ of $\cir{\ell}{t}$ a \textit{$\zeta$-block}.

Apart from concatenation of sequences, we define another sequence combining operation. 
Let $\varphi_1$ and $\varphi_2$ be sequences of lengths $\ell \cdot t$.
A \textit{sequence wreathing of order $\ell$} of $\varphi_1$ and $\varphi_2$, denoted by $\varphi_1 \wreath_\ell \varphi_2$, 
is consecutive concatenation of $k$ subsequent elements of $\varphi_1$ and $\varphi_2$, i.e.
$$
	\varphi_1 \wreath_\ell \varphi_2 = \varphi_1(1,\ell) \ \varphi_2(1,\ell) \dots \varphi_1((t-1)\ell + 1, t \cdot \ell) \ \varphi_2((t-1)\ell+1, t \cdot \ell)\,.
$$
We call the sequences $\varphi_1$ and $\varphi_2$ the \textit{base} and the \textit{wrap} of sequence wrapping $\varphi_1 \wreath_\ell \varphi_2$, respectively.
Additionally, the blocks $\varphi_1(i\ell + 1, (i+1)\ell)$ and $\varphi_2(i\ell + 1, (i+1)\ell)$ are respectively called a \textit{base-block} and a \textit{wrap-block}.

We conclude this section with two lemmas we will use in the forthcoming sections.
The former, due to Currie~\cite{Cur91}, 
states that insertion of non-repetitive subsequences (over distinct alphabets) into a non-repetitive sequence preserves non-repetitiveness.
\begin{lemma}[Currie, 1991]
	\label{lem:insert}
	Let $\varphi_0 = x_1\dots x_t$ be a non-repetitive $\mathbb{A}$-sequence, 
	and $\varphi_1,\dots \varphi_{t+1}$ be non-repetitive $\mathbb{B}$-sequences,
	where $\mathbb{A}$, $\mathbb{B}$ are disjoint alphabets. 
	Additionally, the length of any $\varphi_i$, $1 \le i \le t+1$, may be $0$.
	Then, the sequence $\varphi_1 x_1 \dots \varphi_t x_t \varphi_{t+1}$ is non-repetitive.
\end{lemma}

Proving that a non-repetitive sequence $\varphi$ is $k$-Thue for some integer $k > 1$, 
one needs to show that every $\ell$-subsequence of $\varphi$ is non-repetitive for every integer $\ell$, $1 \le \ell \le k$.
To prove that an $\ell$-subsequence is non-repetitive, it suffices to have enough information about $\varphi$ as we show in the next lemma.
Let $\varphi = \beta_1 \dots \beta_t$ be a sequence comprised of blocks $\beta_i$, $1 \le i \le t$.
We say that a block $\beta_i$ is \textit{uniquely determined by a subset of terms} if there is no block $\beta_j$, $\beta_i \ne \beta_j$,
having the same terms at the same positions.
E.g., from the construction of circular sequences, we have the following.
\begin{observation}
	\label{obs:circular}
	A $\zeta$-block $\varphi^{i}$, $0 \le i \le \ell-1$, is uniquely determined by one term, i.e.,
	given at least one term of a $\varphi^{i}$, one can determine $i$.
\end{observation}

We use the following lemma as a tool for proving that some $d$-subsequence of a Thue sequence does not contain a repetition.
\begin{lemma}
	\label{lem:deter}	
	Let $\sigma$ be an $\ell$-subsequence of a sequence $\varphi = \beta_1 \beta_2 \dots \beta_t$, for some positive integers $\ell$ and $t$. 
	Let $\rho_1 \rho_2$ be a repetition in $\sigma$, and let, for some $j$, $\gamma_1 = \beta_{j+1} \dots \beta_{j+r}$, $\gamma_2 = \beta_{j+r+1} \dots \beta_{j+2r}$
	be the covering sequences of $\rho_1$ and $\rho_2$, respectively.	
	If it holds that
	\begin{itemize}
		\item{} the terms of $\rho_1$ uniquely determine the blocks $\beta_{i}$, for $i \in \set{j+1,j+r}$;
		\item{} the terms of $\rho_2$ uniquely determine the blocks $\beta_{i}$, for $i \in \set{j+r+1,j+2r}$;
		\item{} all the terms of $\rho_1$, $\rho_2$ appear in $\gamma_1$, $\gamma_2$ at the same indices within their blocks, respectively;
	\end{itemize}
	then $\gamma_1 \gamma_2$ is a repetition in $\varphi$.
\end{lemma}

\begin{proof}	
	Since all the blocks are uniquely determined and $r > 0$, it follows that $\beta_{j + i} = \beta_{j + r + i}$ for every $j$.
\end{proof}

%%%%%%%%%%%%%%%%%%%%%%%%%%%%%%%%%%%%%%%%%%%%%%%%%%%%%%%%%%%%%%%%%%%%%%%%%%%%%
\section{Technique \#1: Exhaustive Search for Morphisms}
\label{sec:pas}
%%%%%%%%%%%%%%%%%%%%%%%%%%%%%%%%%%%%%%%%%%%%%%%%%%%%%%%%%%%%%%%%%%%%%%%%%%%%%%

The aim of this section is to present a compact proof of Theorem~\ref{thm:45678}.
For completeness, in the proof, we provide constructions of $k$-Thue sequences also for $k \in \set{2,3}$.
We used an exhaustive computer search to determine appropriate morphisms which are then applied to appropriate sequences.

\begin{proof}[Proof of Theorem~\ref{thm:45678}]
Let $\mathbb{A}_3$ and $\mathbb{A}_{k+2}$ be alphabets on $3$ and $k+2$ letters, respectively.
For every $k$, $2 \le k \le 8$, let a morphism $\mu_k:\mathbb{A}_3^* \to \mathbb{A}_{k+2}^*$ be defined as given below:
\begin{itemize}
	\item{} $k=2$: a $7$-uniform morphism %lim=50
		$$
		\begin{array}{c}
			\mu_2(\texttt{0})=\texttt{0310213}\\
			\mu_2(\texttt{1})=\texttt{0230132}\\
			\mu_2(\texttt{2})=\texttt{0120321}\\
		\end{array}
		$$
	\item{} $k=3$: a $14$-uniform morphism %lim=45
		$$
		\begin{array}{c}
			\mu_3(\texttt{0})=\texttt{10231402310243}\\
			\mu_3(\texttt{1})=\texttt{01243024130243}\\
			\mu_3(\texttt{2})=\texttt{01240312401234}\\
		\end{array}
		$$
	\item{} $k=4$: a $12$-uniform morphism %lim=30
		$$
		\begin{array}{c}
			\mu_4(\texttt{0})=\texttt{012350412534}\\
			\mu_4(\texttt{1})=\texttt{012345103245}\\
			\mu_4(\texttt{2})=\texttt{012340521345}\\
		\end{array}
		$$
	\item{} $k=5$: a $27$-uniform morphism %lim=40
		$$
		\begin{array}{c}
			\mu_5(\texttt{0})=\texttt{012345601235460235146023546}\\
			\mu_5(\texttt{1})=\texttt{012345601234650134625013465}\\
			\mu_5(\texttt{2})=\texttt{012345061234065123460152346}\\
		\end{array}
		$$
	\item{} $k=6$: a $23$-uniform morphism %lim=40
		$$
		\begin{array}{c}
			\mu_6(\texttt{0})=\texttt{01234560172436501243756}\\
			\mu_6(\texttt{1})=\texttt{01234560127354061235476}\\
			\mu_6(\texttt{2})=\texttt{01234560123746510324657}\\
		\end{array}
		$$
	\item{} $k=7$: a $36$-uniform morphism %lim=40
		$$
		\begin{array}{c}
			\mu_7(\texttt{0})=\texttt{012345670812345608721345687201345678}\\
			\mu_7(\texttt{1})=\texttt{012345670182345601872345618702345687}\\
			\mu_7(\texttt{2})=\texttt{012345670128345670281345762801345768}\\
		\end{array}
		$$
	\item{} $k=8$: a $30$-uniform morphism %lim=40
		$$
		\begin{array}{c}
			\mu_8(\texttt{0})=\texttt{012345678902315647890312645789}\\
			\mu_8(\texttt{1})=\texttt{012345678902143675982014365789}\\
			\mu_8(\texttt{2})=\texttt{012345678019324568079123548679}\\
		\end{array}
		$$
\end{itemize}

In what follows, we show that $\mu_k(\varphi')$ is $k$-Thue for every $\paren{\tfrac74}^+$-free sequence $\varphi' \in \mathbb{A}_3^*$.
Using a computer, we have verified the following.
\begin{claim}
	\label{cl:morf}
	Let $\varphi$ be any non-repetitive sequence over $\mathbb{A}_3$ of length at most $40$.
	For each morphism $\mu_k$, $\mu_k(\varphi)$ is $k$-Thue.
\end{claim}

Next, for every $k$ and $d$ such that $2 \le k \le 8$ and $1 \le d \le k$,
we consider every sequence $\delta = x_1x_2x_3x_4$ of length $4$ over $\mathbb{A}_3$ 
and every $d$-subsequence $\sigma$ of $\mu_k(\delta)$ such that $\sigma$ intersects both the prefix $\mu_k(x_1)$ and the suf\mbox{}f\mbox{}ix $\mu_k(x_4)$ of $\mu_k(\delta)$.
We again used a computer to check that if such a $d$-subsequence appears in two sequences $\mu(\delta)$ and $\mu(\delta')$, 
where $\delta=x_1x_2x_3x_4$ and $\delta'=x'_1x'_2x'_3x'_4$, then $x_2x_3=x'_2x'_3$.

Thus, long enough $d$-subsequences of $\mu_k(\varphi)$ allow to determine $\varphi$, except maybe for the first and the last term of $\varphi$.
So, if a large repetition $\rho$ occurs in some $d$-subsequence of $\mu_k(\varphi)$, then $\varphi$ contains a factor $uvu$ such that $u$ is large and $|v|\le 2$.
For $|u| \ge 7$, such a factor $uvu$ cannot appear in a $\paren{\tfrac74}^+$-free sequence. 
On the other hand, if $|u| \le 6$, then the length of $\varphi$ is at most $18$ (including possible first and last term). 
For such sequences, $\mu_k(\varphi)$ are $k$-Thue by Claim~\ref{cl:morf}. 
This completes the proof.
\end{proof}

%%%%%%%%%%%%%%%%%%%%%%%%%%%%%%%%%%%%%%%%%%%%%%%%%%%%%%%%%%%%%%%%%%%%%%%%%%%%%
\section{Construction of Thue sequences using Hexagonal Morphism}
\label{sec:koc}
%%%%%%%%%%%%%%%%%%%%%%%%%%%%%%%%%%%%%%%%%%%%%%%%%%%%%%%%%%%%%%%%%%%%%%%%%%%%%%

Recently, in his master thesis, Ko\v{c}i\v{s}ko~\cite{Koc13} introduced a uniform morphism $\kappa$, 
which maps a term $x$ of a sequence to a block of three symbols regarding the mapping of the predecessor of $x$.
In particular, instead of using an alphabet $\mathbb{A} = \set{1,2,3}$ an auxiliary alphabet 
$$
	\overline{\mathbb{A}} = \set{\overline{1}, \underline{1}, \overline{2}, \underline{2}, \overline{3}, \underline{3}}
$$ 
is used. The morphism $\kappa$ is then defined as
\begin{eqnarray*}
	\kappa(\overline{1}) = \overline{1} \ \underline{2} \ \overline{3}\,, 
		\quad \quad 
	\kappa(\overline{2}) = \overline{2} \ \underline{3} \ \overline{1}\,, 	
		\quad \quad 
	\kappa(\overline{3}) = \overline{3} \ \underline{1} \ \overline{2}\,, \\
	\kappa(\underline{1}) = \underline{3} \ \overline{2} \ \underline{1}\,, 
		\quad \quad 
	\kappa(\underline{2}) = \underline{1} \ \overline{3} \ \underline{2}\,,
		\quad \quad 
	\kappa(\underline{3}) = \underline{2} \ \overline{1} \ \underline{3}\,.
\end{eqnarray*}
For a positive integer $t$, we recursively define the sequence
$$
	\skap{t} = \kappa(\skap{t-1})\,,
$$
where $\skap{0} = \overline{1}$.
Notice that for every $t$, every symbol from $\overline{\mathbb{A}}$ is a neighbor of at most two symbols of $\overline{\mathbb{A}}$ (if $t > 3$, then precisely two);
we say that neighboring symbols are \textit{adjacent}. 
The adjacency is also preserved between the blocks of three symbols to which the
symbols from $\overline{\mathbb{A}}$ are mapped by $\kappa$; 
we denote these blocks \textit{$\overline{\kappa}$-triples}. 
Due to its structure, we refer to $\kappa$ as the \textit{hexagonal morphism}.
In Fig.~\ref{fig:koci}, the adjacencies between the symbols and the $\overline{\kappa}$-triples,
and the mappings of $\kappa$ are depicted.
\begin{figure}[htp!]
	$$
		\includegraphics{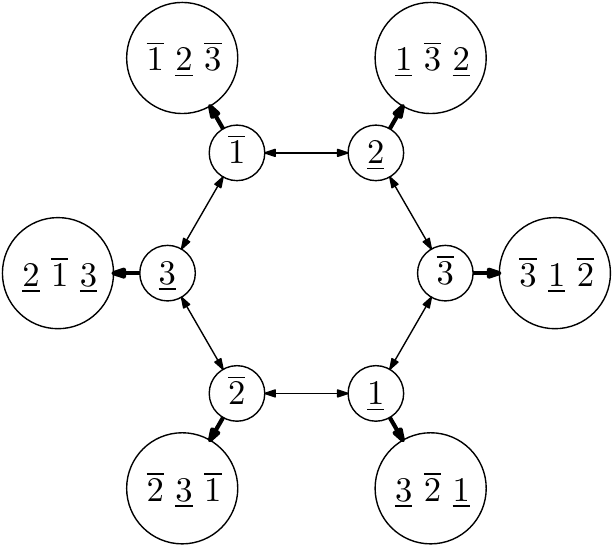}
	$$
	\caption{The graph of adjacencies between the symbols of $\overline{\mathbb{A}}$ and $\overline{\kappa}$-triples, and the mappings defined by $\kappa$.}
	\label{fig:koci}
\end{figure}

Let $\pi \ : \ \overline{\mathbb{A}} \rightarrow \mathbb{A}$ be a projection of symbols from the auxiliary alphabet
$\overline{\mathbb{A}}$ to $\mathbb{A}$ defined as $\pi(\overline{a}) = a$ and $\pi(\underline{a}) = a$, for every $a \in \set{1,2,3}$.
By $\kap{t}$, we denote the projected sequence $\skap{t}$, i.e. $\kap{t} = \pi(\skap{t})$; similarly a projected $\overline{\kappa}$-triple $\tau$, 
$\pi(\tau)$, is referred to as a $\kappa$-block.

By the definition of $\kap{t} = \set{x_i}_{i=1}^{3^t}$ and the mapping $\kappa$, one can easily derive the following basic properties:
\begin{itemize}
	\item[$(K_1)$] \quad For every pair of adjacent $\kappa$-blocks $\tau$ and $\sigma$, the sequence $\tau\sigma$ is Thue.
	\item[$(K_2)$] \quad The length of $\kap{t}$ is $3^t$, and $x_{3i+1}x_{3i+2}x_{3i+3}$ is a $\kappa$-block for every $i$, $0 \le i < 3^{t-1}$.
	\item[$(K_3)$] \quad $\set{x_{3i+1}, x_{3i+2}, x_{3i+3}} = \set{1,2,3}$ for every $i$, $0 \le i < 3^{t-1}$.
	\item[$(K_4)$] \quad $x_{3i+2} \neq x_{3(i+1)+2}$ for every $i$, $0 \le i < 3^{t-1}-1$.
	\item[$(K_5)$] \quad Any three consecutive terms $x_{j+1}x_{j+2}x_{j+3}$ of $\kap{t}$, which do not belong to the same $\kappa$-block,
					uniquely determine the two $\kappa$-blocks they belong to.
	\item[$(K_6)$] \quad For a pair $\tau_1$, $\tau_2$ of adjacent $\kappa$-blocks it holds that the first term of $\tau_1$ is distinct 
					from the third term of $\tau_2$.
	\item[$(K_7)$] \quad If a pair of distinct $\kappa$-blocks has the same first or last term, then they are adjacent.
	\item[$(K_8)$] \quad A pair of adjacent $\kappa$-blocks is not adjacent to any other common $\kappa$-block.
	\item[$(K_9)$] \quad The middle term of the $\kappa$-block $\pi(\kappa(i))$, $i \in \overline{\mathbb{A}}$, equals $\pi(i)+1$ (modulo $3$).
	\item[$(K_{10})$] \quad A pair of distinct $\kappa$-symbols $x_1$ and $x_2$, where $x_1$ and $x_2$ are the first (last) terms of adjacent $\kappa$-blocks $\tau_1$ 
					and $\tau_2$, uniquely determines $\tau_1$ and $\tau_2$.
	\item[$(K_{11})$] \quad A $\kappa$-block $\tau_1$ and at least one term of a $\kappa$-block $\tau_2$ adjacent to $\tau_1$ uniquely determine $\tau_2$.
	\item[$(K_{12})$] \quad A pair of adjacent $\kappa$-blocks is in $\kap{t}$ always separated by an even number of $\kappa$-blocks, since the graph of 
				adjacencies is bipartite.
\end{itemize}

We use (some of) the properties above, to prove the following theorem.
\begin{theorem}[Ko\v{c}i\v{s}ko, 2013]
	\label{thm:koci}
	The sequence $\kap{t}$ is Thue, for every non-negative integer $t$.
\end{theorem}
For the sake of completeness, we present a short proof of Theorem~\ref{thm:koci} here also.
\begin{proof}
	We prove the theorem by induction. Clearly, $\kap{0}$ is Thue. 
	Consider the sequence $\kap{t} = \set{x_i}_{i=1}^{3^t}$ and suppose that $\kap{j}$ is Thue for every $j < t$.
	Suppose for a contradiction that there is a repetition in $\kap{t}$ and 
	let $\rho_1\rho_2 = y_1 \dots y_r y_{r+1} \dots y_{2r}$ be a repetition with the minimum length
	(for later purposes we distinguish two repetition factors, although $\rho_1 = \rho_2$).
	By $(K_1)$, we have that $r \ge 3$. We consider two subcases regarding the length $r$ of $\rho_1 (=\rho_2)$.
	
	Suppose first that $r$ is divisible by $3$. Then, as we show in the following claim,
	we may assume that the term $y_1$ is the first term of some $\kappa$-block.
	\begin{claim}
		\label{cl:koci1}
		Let $r$ be divisible by $3$. If $y_1 = x_{3i+2}$ (resp. $y_1 = x_{3i+3}$) for some $i$, $0 \le i < 3^{t-1}$, 
		then $x_{3i+1} x_{3i+2} \dots x_{3i + 2r}$ (resp. $x_{3(i+1)+1} x_{3(i+1)+2} \dots x_{3(i+1) + 2r}$) is also a repetition.
	\end{claim}
	\begin{proofclaim}
		Suppose that $y_1 = x_{3i+2}$. By $(K_3)$, every $\kappa$-block is uniquely determined by two symbols. So $x_{3i+1} = y_r$
		and hence $x_{3i+1} \dots x_{3i + 2r} = y_r y_1 \dots y_{2r-1}$ is a repetition.		
		A proof for the case $y_1 = x_{3i+3}$ is analogous.
	\end{proofclaim}
	
	Hence, we have that $\rho = \tau_1 \dots \tau_{\frac{r}{3}}$, 
	where $\tau_j$ are $\kappa$-blocks for every $j$, $1 \le j \le \frac{r}{3}$.
	But in this case, there is a repetition already in $\kap{t-1}$, contradicting the induction hypothesis.
	
	Therefore, we may assume that $r$ is not divisible by $3$. 
	This means that the first terms $y_1$ and $y_{r+1}$ of the two repetition factors $\rho_1$ and $\rho_2$, respectively, 
	are at different positions within the $\kappa$-blocks they belong to. For example, if $r = 3k+1$, and $y_1$ is the first term
	of the $\kappa$-block $y_1 y_2 y_3$, then $y_{r+1}$ is the second term of the $\kappa$-block $y_r y_{2r+1} y_{2r+2}$.
	There are hence six possible cases regarding the position of $y_1$ and $y_{r+1}$ in their $\kappa$-blocks.	
	
	Suppose first that $y_1$ is the first term of the $\kappa$-block $x_1 x_2 x_3$. 
	By $(K_3)$, $x_1$, $x_2$, and $x_3$ are pairwise distinct. 
	Since $\rho_1 = \rho_2$, we thus know three consecutive elements of two $\kappa$-blocks 
	(the one of $y_{r+1}$ and the subsequent one). 
	By $(K_5)$, we can determine both $\kappa$-blocks, which gives us information about the term $y_4$. 
	Using $(K_5)$ again, we can determine the $\kappa$-block $y_4 y_5 y_6$, 
	namely $y_4 y_5 y_6 = x_2 x_1 x_3$ in the case when $r \equiv 1 \bmod{3}$, 
	and $y_4 y_5 y_6 = x_1 x_3 x_2$ in the case when $r \equiv 2 \bmod{3}$.
	Using the information obtained by determining $\kappa$-blocks using $(K_5)$, 
	we infer that every $\kappa$-block of $\rho_1$ ends with $x_3$ in the former case, or 
	starts with $x_1$ in the latter case. As $\rho_1$ and $\rho_2$ are concatenated, 
	this leads us to contradiction on the existence of a repetition.
	With a similar argument, we obtain a contradiction in the case when $y_{r+1}$ is the first term of its $\kappa$-block.
		
	Suppose now that $r = 3k+1$, for some positive integer $k$, 
	and $y_1$ is the second term of its $\kappa$-block, say $x_3 x_1 x_2$. 
	Then $y_{r+1} = x_1$ and $y_{r+2} = x_2$, where $y_{r+1}$ and $y_{r+2}$ belong to distinct $\kappa$-blocks.
	Notice that there are two possibilities for the value of $y_{r+3}$, namely $x_1$ and $x_3$.
	However, regardless the choice, after determining the $\kappa$-block $y_{r+2} y_{r+3} y_{r+4}$ by $(K_5)$, 
	and continue by alternately determining $\kappa$-blocks in $\rho_1$ and $\rho_2$, as described above,
	we infer that in both cases, every $\kappa$-block in $\rho_1$ ends with $x_2$, a contradiction.
	An analogous analysis may be performed in the last case, when $r = 3k+2$ and $y_1$ being the third term of its $\kappa$-block.	
\end{proof}

%%%%%%%%%%%%%%%%%%%%%%%%%%%%%%%%%%%%%%%%%%%%%%%%%%%%%%%%%%%%%%%%%%%%%%%%%%%%%%
\section{Technique \#2: Transposition \& Cyclic Blocks}
\label{sec:cons}
%%%%%%%%%%%%%%%%%%%%%%%%%%%%%%%%%%%%%%%%%%%%%%%%%%%%%%%%%%%%%%%%%%%%%%%%%%%%%%

In this section, we present alternative proofs to answer Conjecture~\ref{conj:k+2} in affirmative for the cases $k = 4$ and $k = 6$.
For each of the two cases we present a special morphism and apply it on a non-repetitive sequence. Then, we use sequence wreathing to
extend the sequence by circular blocks. 

\subsection{The case $k=4$}
\label{sub:4}

In this part, to prove the case $k=4$ in Theorem~\ref{thm:45678}, 
we combine the sequence $\kap{t}$ obtained by the hexagonal morphism and the circular sequence $\cir{3}{3^t}$ by wreathing. 
We construct $\kap{t}$ over the alphabet $\set{1,2,3}$, and $\cir{3}{3^{t-1}}$ over the alphabet $\set{4,5,6}$.
We define
$$
	\varphi_4^t = \kap{t} \wreath_3 \cir{3}{3^{t-1}}\,.
$$
For clarity, we refer to the base-blocks of $\varphi_4^t$ as \textit{$\kappa$-blocks} (recall that the wrap-blocks are called \textit{$\zeta$-blocks}).
Additionally, the terms from $\kappa$-blocks (resp. $\zeta$-blocks) are called $\kappa$-terms (resp. $\zeta$-terms). 
\begin{lemma}
	\label{lem:4thue}
	The sequence $\varphi_4^t$ is $4$-Thue for every non-negative integer $t$.
\end{lemma}

\begin{proof}
	Since $\kap{t}$ is Thue by Theorem~\ref{thm:koci}, $\varphi_4^t$ is also Thue by Lemma~\ref{lem:insert}. 
	Thus, it remains to prove that every $d$-subsequence of $\varphi_4^t$ is Thue, for every $d \in \set{2,3,4}$. 
	Observe first that by $(K_1)$, $(K_6)$, and the definition of circular sequences, 
	every five consecutive terms of $\varphi_4^t$ are distinct.
	This in particular means that 
	\begin{itemize}
		\item[$(P_1)$] \quad \textit{there are no repetitions of length $2$ or $4$ in any $d$-subsequence of $\varphi_4^t$.}
	\end{itemize}
	Moreover, 
	\begin{itemize}
		\item[$(P_2)$] \quad \textit{in every $d$-subsequence of $\varphi_4^t$ there are at most two consecutive $\kappa$-terms or $\zeta$-terms;}
		\item[$(P_3)$] \quad \textit{in every $d$-subsequence of $\varphi_4^t$ any repetition contains $\kappa$-terms and $\zeta$-terms;}
		\item[$(P_4)$] \quad \textit{if a $\kappa$-term (resp. $\zeta$-term) in a $d$-subsequence $\sigma$ of $\varphi_4^t$, 
				whose predecessor and successor in $\sigma$ are $\zeta$-terms (resp. $\kappa$-terms), is at index $i$ within its $\kappa$-block (resp. $\zeta$-block),
				then every $\kappa$-term (resp. $\zeta$-term) in $\sigma$ is at index $i$ within its $\kappa$-block (resp. $\zeta$-block).}
	\end{itemize}
	All the latter three properties are direct corollaries of $(P_1)$ and the fact that every $\kappa$-block and $\zeta$-block is of length $3$.	
	
	Now, we prove that every $d$-subsequence of $\varphi_4^t$ is non-repetitive, considering three cases with regard to $d$.
	In each case, we assume there is a repetition $\rho_1 \rho_2 = y_1 \dots y_{r} y_{r+1} \dots y_{2r}$ in 
	some $d$-subsequence $\sigma$ and eventually reach a contradiction on its existence.	
	
	By $(P_3)$, there is at least one $\zeta$-term in $\rho_1$. 
	Moreover, by the definition of circular sequences and $\varphi_4^t$,
	every three consecutive $\zeta$-terms in $\rho_1$ (ignoring the $\kappa$-terms) are distinct, 
	unless $d=4$ and the $\zeta$-terms of $\rho_1$ are at indices $1$ and $3$ in $\zeta$-blocks. 
	However, in such a case, by construction of circular sequences, without loss of generality,
	consecutive $\zeta$-terms of $\rho_1$ are $4\ 4\ 5\ 5\ 6\ 6\dots$, which means that $r$
	must be divisible by $6$, to have the same sequence of $\zeta$-terms in $\rho_2$.
	This implies that  
	\begin{itemize}
		\item[$(P_5)$] \quad \textit{the number of $\zeta$-terms in $\rho_1$ is divisible by $3$,}
	\end{itemize}
	and consequently, since in $\zeta$-blocks the symbols repeat at the same indices in every third block:
	\begin{itemize}
		\item[$(P_6)$] \quad \textit{the number of $\zeta$-blocks to which the $\zeta$-terms of $\rho_1$ belong to in $\varphi_4^t$ is divisible by $3$.}
	\end{itemize}
	Observe that, by the above properties, 
	\begin{itemize}
		\item[$(P_7)$] \quad \textit{the first terms of $\rho_1$ and $\rho_2$ are either both $\kappa$-terms or $\zeta$-terms, and moreover, 
			they appear at the same index within their blocks in $\varphi_4^t$.}
	\end{itemize}
	Now, we start the analysis regarding $d$:
	\begin{itemize}
		\item{} \textit{$d = 2$.} \\
			Suppose first that $y_1$ is the first term of some $\kappa$-block.
			Then, $\rho_1$ is comprised alternately of two $\kappa$-terms (the first and the third terms of a $\kappa$-block in $\varphi_4^t$)
			and one $\zeta$-term (the second term of its $\zeta$-block in $\varphi_4^t$).
			Consequently, $y_{r+1}$ is the first term of a $\kappa$-block also, 
			and	the last term of $\rho_1$ must be a $\zeta$-term. 
			By $(K_3)$, every $\kappa$-block is uniquely determined by two of its terms, hence
			one can determine all $\kappa$-blocks to which the $\kappa$-terms of $\rho_1$ and $\rho_2$ belong to in $\varphi_4^t$. 
			Similarly, all the $\zeta$-blocks, to which $\zeta$-terms of $\rho_1$ and $\rho_2$ belong, are uniquely determined by Observation~\ref{obs:circular}. 
			Moreover, since the terms $y_{r}$ and $y_{r+1}$ belong to different
			blocks, we can apply Lemma~\ref{lem:deter} obtaining a contradiction on the existence of $\rho_1 \rho_2$.
			
			Suppose now that $y_1$ is the third term of some $\kappa$-block.
			A similar argument as in the paragraph above shows that $y_r$ 
			is the first term of some $\kappa$-block $\gamma_r$ of $\varphi_4^t$,
			while $y_{r+1}$ is the third term of $\beta_r$. Note that the terms $y_2,y_3$ and $y_{r+2},y_{r+3}$ uniquely determine the same block $\gamma_2$.
			Consider now the $\kappa$-block $\gamma_{1}$ to which $y_1$ belongs. 
			By $(K_7)$, it is one of the two possible $\kappa$-blocks that end with $y_1$, and since $\gamma_1$ and $\gamma_r$ are adjacent to $\gamma_2$,
			by $(K_8)$, we infer that $\gamma_1 = \gamma_r$. 
			Thus, taking the first term $z$ of $\gamma_1$, we have a repetition $z y_1 \dots y_{r-1} y_r \dots y_{2r-1}$ in $\sigma$, 
			which satisfies the assumptions of Lemma~\ref{lem:deter}. 
			Hence, there is a repetition in $\varphi_4^t$, a contradiction.
			
			Next, suppose that $y_1$ is the second term of some $\kappa$-block.
			Then, by $(P_4)$, all the $\kappa$-terms in $\rho_1$ and $\rho_2$ are the second terms of $\kappa$-blocks in $\varphi_4^t$.
			By $(K_9)$, we have that the second terms of $\kappa$-blocks in $\varphi_4^t$ are exactly the terms of $\kap{t-1}$ shifted by $1$,
			and thus form a non-repetitive sequence. 
			Using Lemma~\ref{lem:insert}, we infer that the sequence $\sigma$ is also non-repetitive, a contradiction.
			
			Finally, suppose that $y_1$ is a $\zeta$-term. 
			Let $\gamma$ be the last $\zeta$-block in $\varphi_4^t$ to which some term of $\rho_2$ belongs.
			Since a $\zeta$-block is uniquely determined by at least one of its terms, using $(P_5)$, we infer that the 
			$\zeta$-block of $\varphi_4^t$ following $\gamma$ is equal to the $\zeta$-block uniquely determined by $y_1$.
			Let $y \in \set{y_2, y_3}$ be the first $\kappa$-term of $\rho_1$. 
			The observation above implies that there exists a repetition in $\sigma$ starting with $y$ and ending with the $\zeta$-terms before $y$ in $\rho_1$.
			Such a repetition cannot exist due to the analysis of the cases above.
		\item{} \textit{$d = 3$.} \\
			Suppose that $y_1$ is the first term of some $\kappa$-block.
			Clearly, the first term of $\rho_2$ is also a $\kappa$-term, and thus the number of $\kappa$-terms 
			in $\rho_1$ is divisible by $3$, by $(P_5)$. 
			Therefore, there are at least six $\kappa$-terms in $\rho_1 \rho_2$, 
			meaning there are two distinct consecutive $\kappa$-terms. 
			Using $(K_{10})$ and $(K_{11})$, we can uniquely determine all $\kappa$-blocks to which the $\kappa$-terms of $\rho_1$ and $\rho_2$ belong. 
			So, by Lemma~\ref{lem:deter}, we obtain a contradiction.
			
			If $y_1$ is the third term of some $\kappa$-block, we use the same argument as in the paragraph above.
			
			The argument when $y_1$ is the second term of some $\kappa$-block is analogous to the subcase in the case $d=2$, 
			where $y_1$ is the second term of some $\kappa$-block.
			
			In the case when $y_1$ is a $\zeta$-term, we can again translate the analysis to the one of the above
			cases, since the $\zeta$-triples have period $3$.
		\item{} \textit{$d = 4$.} \\
			Suppose that $y_1$ is the first term of some $\kappa$-block. 
			By $(P_1)$, the length of $\rho_1$ is at least $3$.
			Furthermore, $y_2$ and $y_3$ are a $\zeta$-terms (the second terms of some $\zeta$-block)
			and a $\kappa$-term (the third term of some $\kappa$-block), respectively.
			By $(P_4)$, all $\zeta$-terms in $\rho_1$ are the second terms of $\zeta$-blocks.
			Thus, by $(P_5)$ and the fact that for every $\zeta$-term in $\rho_1$ there are two $\zeta$-blocks in $\hat{\rho_1}$,
			we have that the number of $\zeta$-blocks in $\hat{\rho_1}$ is divisible by $6$.
			By $(P_7)$, $y_{r+1}$ is also the first term of some $\kappa$-block in $\varphi_4^t$, 
			meaning that the number of $\kappa$-blocks in $\hat{\rho_1}$ is also divisible by $6$ 
			and that the number of $\kappa$-blocks between the blocks of $y_1$ and $y_{r+1}$ is odd. 
			Hence, by $(K_{12})$, the $\kappa$-blocks of $y_1$ and $y_{r+1}$ are the same.
			Analogously, all the blocks of the $\kappa$-terms $y_i$ in $\rho_1$ 
			are the same as the $\kappa$-blocks of $y_{i+r}$ in $\rho_2$.
			Thus, there is a repetition in $\varphi_4^t$ also, a contradiction.
			
			Suppose now that $y_1$ is the second term of some $\kappa$-block.
			Similarly as in the case above, we notice that all $\kappa$-terms of $\rho_1$ are the second terms in their $\kappa$-blocks
			in $\hat{\rho_1}$, and that the number of $\kappa$-blocks in $\hat{\rho_1}$ is divisible by $6$.
			Again, we deduce that for every two $\kappa$-terms $y_i$ and $y_j$ in $\sigma$, there are even number of $\kappa$-blocks between the 
			$\kappa$-blocks of $y_i$ and $y_j$ in $\hat{\sigma}$. 
			It follows that every pair of equal $\kappa$-symbols in $\sigma$ belongs
			to the same $\kappa$-block, and hence $\hat(\rho_1) = \hat{\rho_2}$, a contradiction.
			
			The cases, when $y_1$ is the third term of some $\kappa$-block, or the second term of some $\zeta$-block 
			are analogous to the first case. 
			The cases, when $y_1$ is the first or the third term of some $\zeta$-block are analogous to the second case.
	\end{itemize}	
	
\end{proof}

\subsection{The case $k=6$}

In this part, we present a construction of a $6$-Thue sequence using $8$ symbols, 
in a similar way as for the case $k=4$. Again, we wreath a Thue sequence with a circular sequence,
but now, the base sequence is formed by blocks of four symbols, 
where in each block we only permute symbols in fixed pairs.

Similarly as in Section~\ref{sec:koc}, we start by constructing a Thue sequence over an alphabet 
$$
	\mathbb{B} = \set{1,2,3,4}
$$ 
of $4$ symbols. Let a morphism $\lambda$, mapping a symbol from the sequence to a block of four distinct symbols, be defined as
\begin{eqnarray*}
	\lambda(1) = 1 \ 2 \ 3 \ 4 \,, 
		\quad  
	\lambda(2) = 2 \ 1 \ 4 \ 3 \,,
		\quad
	\lambda(3) = 1 \ 2 \ 4 \ 3 \,, 
		\quad  
	\lambda(4) = 2 \ 1 \ 3 \ 4 \,.
\end{eqnarray*}
For a positive integer $t$, we recursively define the sequence
$$
	\slam{t} = \lambda(\slam{t-1})\,,
$$
where $\slam{0} = 1$. 
Notice that for every positive integer $t$, 
every symbol from $\mathbb{B}$ is a neighbor of all symbols of $\mathbb{B}$.
The blocks of four symbols to which the symbols from $\mathbb{B}$ are mapped by $\lambda$, 
are referred to as \textit{$\lambda$-blocks}. 
In Fig.~\ref{fig:k6}, the mappings of $\lambda$ are depicted.
\begin{figure}[htp]
	$$
		\includegraphics{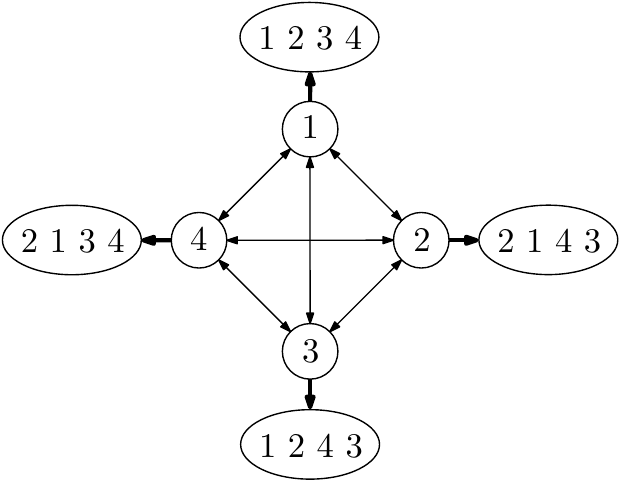}
	$$
	\caption{The graph of adjacencies between the symbols of $\mathbb{B}$ and $\lambda$-blocks, and the mappings defined by~$\lambda$.}
	\label{fig:k6}
\end{figure}

We first observe some basic properties of the sequence $\slam{t}$, for any positive integer $t$.
\begin{itemize}
	\item[$(L_1)$] \quad For any pair of adjacent $\lambda$-blocks $\gamma_1$ and $\gamma_2$, the sequence $\gamma_1 \gamma_2$ is Thue.
	\item[$(L_2)$] \quad The length of $\slam{t}$ is $4^t$, and $x_{4i+1}x_{4i+2}x_{4i+3}x_{4i+4}$ is a $\lambda$-block for every $i$, $0 \le i \le 4^{t-1}-1$.
	\item[$(L_3)$] \quad $\{x_{4i+1},x_{4i+2}\} = \set{1,2}$ and $\{x_{4i+3},x_{4i+4}\} = \{ 3,4 \}$ for every $i$, $0 \le i \le 4^{t-1}-1$.
		Consequently, by knowing at least one term at index $1$ or $2$, and at least one term at index $3$ or $4$, 
		the $\lambda$-block is uniquely determined.
	\item[$(L_4)$] \quad For every $i$, $0 \le i \le 4^{t-1} - 3$, it holds: $x_{4i+1}x_{4i+2}x_{4i+3}x_{4i+4} \neq x_{4i+9}x_{4i+10}x_{4i+11}x_{4i+12}$ 
		(this is in fact a consequence of $(L_3)$).
	\item[$(L_5)$] \quad Two consecutive $\lambda$-blocks with the same first two terms are mapped from $\set{1,3}$ or $\set{2,4}$.
		Similarly, two consecutive $\lambda$-blocks with the same last two terms are mapped from $\set{1,4}$ or $\set{2,3}$.
	\item[$(L_6)$] \quad Let $\gamma_1$ and $\gamma_2$ be distinct $\lambda$-blocks with equal terms 
		at indices $1$ and $2$ or at indices $3$ and $4$.
		For $\lambda$-blocks $\gamma_3$, $\gamma_4$, and $\gamma_5$, 
		in $\slam{t}$, there is at most one of the subsequences $\gamma_1 \gamma_3 \gamma_5$ and $\gamma_2 \gamma_4 \gamma_5$, 
		since otherwise the property $(L_3)$ would be violated in $\slam{t-1}$.
	\item[$(L_7)$] \quad If for two $\lambda$-blocks $\gamma_1 = x_{4i+1}x_{4i+2}x_{4i+3}x_{4i+4}$ and $\gamma_2 = x_{4j+1}x_{4j+2}x_{4j+3}x_{4j+4}$
		there is such $\ell \in \set{1,2,3,4}$ that $x_{4i+\ell} = x_{4j+\ell}$ and $4$ divides $|j-i|$, then $\gamma_1 = \gamma_2$.
		On the other hand, if $4$ does not divide $|j-i|$, but $|j-i|$ is even, then $\gamma_1 \neq \gamma_2$.
	\item[$(L_8)$] \quad If for a $\lambda$-block $\gamma$ one term is known, then it is one of two possible $\lambda$-blocks.
		In particular, if the known term is at index $1$ or $2$ in $\gamma$, then either $\lambda^{-1}(\gamma) \in \set{1,3}$ or $\lambda^{-1}(\gamma) \in \set{2,4}$.
		If the known term is at index $3$ or $4$ in $\gamma$, then either $\lambda^{-1}(\gamma) \in \set{1,4}$ or $\lambda^{-1}(\gamma) \in \set{2,3}$.
\end{itemize}
We leave the above properties to the reader to verify and proceed by proving that $\slam{t}$ is Thue.

\begin{lemma}
	\label{lem:k6}
	The sequence $\slam{t}$ is Thue for every non-negative integer $t$.
\end{lemma}

\begin{proof}
	Suppose the contrary, and let $t$ be the minimum such that there is a repetition in $\slam{t}$.
	Denote the $i$-th term of $\slam{t}$ by $x_i$.
	Let $\rho_1 \rho_2 = y_1\dots y_r y_{r+1}\dots y_{2r}$ be a repetition of minimum length.
	We first show that $r > 4$. The cases with $r \le 3$ are trivial, so suppose $r=4$. 
	By $(L_1)$, we have that $y_1$ is not at index $4i+1$ in $\slam{t}$ (for any $i$, $0\le i \le 4^{t-1}-1$),
	and by $(L_3)$, it is not at index $4i+2$ nor $4i+4$. 
	Hence, assume $y_1$ is at index $4i+3$. 
	Denote the $\lambda$-block $x_{4i+5}x_{4i+6}x_{4i+7}x_{4i+8} (= y_3y_4y_5y_6)$ by $\gamma_1$.
	By $(L_1)$, we have that $\gamma_0 = x_{4i+1}x_{4i+2}x_{4i+3}x_{4i+4} = y_4y_3y_5y_6$ and similarly,
	$\gamma_2 = x_{4i+9}x_{4i+10}x_{4i+11}x_{4i+12} = y_3y_4y_6y_5$. By $(L_5)$, 
	this means that if $\gamma_1 \in \set{1,2}$, then $\gamma_0, \gamma_2 \in \set{3,4}$, and analogously, 
	if $\gamma_1 \in \set{3,4}$, then $\gamma_0, \gamma_2 \in \set{1,2}$, a contradiction to $(L_3)$.
	Hence, $r > 4$.
	
	Let $j$ be the index of $y_1$ in $\slam{t}$, i.e. $y_1 = x_j$.
	If $j$ is odd, then by $(L_3)$, either $x_{j}x_{j+1} = \set{1,2}$ or $x_{j}x_{j+1} = \set{3,4}$,
	and without loss of generality, we may assume the former.
	Thus, also $x_{j+r}x_{j+r+1} = \set{1,2}$, which implies that $r$ must be even. 
	In the case when $j$ is even, $(L_3)$ similarly implies that $x_j \in \set{1,2}$ and $x_{j+1} \in \set{3,4}$,
	and hence $x_{j+r} \in \set{1,2}$ and $x_{j+r+1} \in \set{3,4}$. Consequently, $r$ is again even.
	Finally observe that by $(L_2)$, from $r$ being even and $x_j = x_{j+r}$ it follows that $r$ is divisible by $4$.
	
	Suppose now that $j = 4i+1$, for some $i$. 
	Then, since $r$ is divisible by $4$, $\rho_1$ and $\rho_2$ are comprised of $\tfrac{r}{4}$ $\lambda$-blocks each, 
	the first starting with $x_j$.
	This in turn means that there is a repetition in $\slam{t-1}$ as every $\lambda$-block represents one term in $\slam{t-1}$,
	a contradiction to the minimality of $t$.
	
	Next, suppose $j = 4i + 2$. By $(L_3)$, we have that $x_{j-1} = x_{j+r-1}$, and hence 
	$\rho_1' \rho_2' = x_{j-1}x_{j}\dots x_{j+r-1} x_{j+r} \dots x_{j+2r-2}$ is also a repetition in $\slam{t}$,
	where $j-1 = 4i + 1$, and hence the reasoning in the above paragraph applies.
	
	Suppose $j = 4i + 4$. Then, analogous to the previous case, we infer $x_{j+r} = x_{j+2r}$, and hence 
	$\rho_1' \rho_2' = x_{j+1}\dots x_{j+r} x_{j+r+1} \dots x_{j+2r}$ is also a repetition in $\slam{t}$,
	where $j+1 = 4(i+1) + 1$, so the reasoning for $j = 4i+1$ applies again.
	
	Finally, consider the case with $j = 4i + 3$. If $r = 8$, from $x_{j+2}x_{j+3}x_{j+4}x_{j+5} = x_{j+10}x_{j+11}x_{j+12}x_{j+13}$
	it follows that the $\lambda$-block $x_{j+6}x_{j+7}x_{j+8}x_{j+9}$ is surrounded by the same $\lambda$-blocks, which 
	contradicts $(L_4)$. 
	Hence, we may assume $r \ge 12$. Since the $\lambda$-blocks $x_{j+6}x_{j+7}x_{j+8}x_{j+9}$ and $x_{j+r+6}x_{j+r+7}x_{j+r+8}x_{j+r+9}$ are equal,
	and $r$ is divisible by $4$, it follows that also $x_{j-2}x_{j-1}x_{j}x_{j+1} = x_{j+r-2}x_{j+r-1}x_{j+r}x_{j+r+1}$ 
	and we may apply the reasoning for the case with $j=4i+1$ on the repetition $x_{j-2}\dots x_{j+2r-3}$.	
	Hence, $\slam{t}$ is Thue.
\end{proof}

Now, take the circular sequence $\cir{4}{4^{t-1}}$, with $\varphi = 5 \ 6 \ 7 \ 8$,
and use sequence wreathing on $\slam{t}$ and $\cir{4}{4^{t-1}}$ to obtain the sequence
$$
	\varphi_6^t = \slam{t} \wreath_4 \cir{4}{4^{t-1}}\,.
$$
Similarly as above, we refer to the base-blocks of $\varphi_6^t$ as $\lambda$-blocks, and to the wrap-blocks as $\zeta$-blocks.
The terms of $\lambda$-blocks (resp. $\zeta$-blocks) are referred to as $\lambda$-terms (resp. $\zeta$-terms).
The sequence $\varphi_6^2$ is hence:
$$
	\underbrace{1 \ 2 \ 3 \ 4}_{\lambda(1)} \ \ 5 \ 6 \ 7 \ 8 \ \ \ \underbrace{2 \ 1 \ 4 \ 3}_{\lambda(2)} \ \ 6 \ 7 \ 8 \ 5 \ \ \ \underbrace{1 \ 2 \ 4 \ 3}_{\lambda(3)} \ \ 7 \ 8 \ 5 \ 6 \ \ \underbrace{2 \ 1 \ 3 \ 4}_{\lambda(4)} \ \ 8 \ 5 \ 6 \ 7
$$

It remains to prove that $\varphi_6^t$ is also $6$-Thue.
\begin{lemma}
	\label{lem:6k}
	The sequence $\varphi_6^t$ is $6$-Thue for every non-negative integer $t$.
\end{lemma}

\begin{proof}
	By Lemmas~\ref{lem:insert} and~\ref{lem:k6}, we have that $\varphi_6^t$ is Thue. 
	Thus, we only need to prove that every $d$-subsequence of $\varphi_6^t$ is also Thue, for every $d \in \set{2,3,4,5,6}$.
	First, we list some general properties and then consider $d$-subsequences separately regarding the values of $d$.
	
	By $(L_3)$ and the definition of circular sequences, every seven consecutive terms of $\varphi_6^t$ are distinct.
	Hence,
	\begin{itemize}
		\item[$(R_1)$] \quad there are no repetitions of length $2$ or $4$ in any $d$-subsequence $\varphi_6^t$.		
	\end{itemize}
	
	Furthermore, since the length of any $\lambda$-block and $\zeta$-block in $\varphi_6^t$ is $4$, one can deduce that:
	\begin{itemize}
		\item[$(R_2)$] \quad \textit{in every $d$-subsequence of $\varphi_6^t$ there are at most two consecutive $\lambda$-terms or $\zeta$-terms;}
		\item[$(R_3)$] \quad \textit{in every $d$-subsequence of $\varphi_6^t$ any repetition contains $\lambda$-terms and $\zeta$-terms;}
	\end{itemize}
	
	Given a $d$-subsequence $\sigma = z_1 z_2 \dots z_n$ of $\varphi_6^t$ consisting of $n$ elements, we define a 
	mapping $\vartheta : \Sigma \rightarrow \set{N,C}^n$, where $\Sigma$ represents the set of all $d$-subsequences of $\varphi_6^t$,
	mapping $\sigma$ to an $n$-component vector, $i$-th component being $N$ if $z_i$ belongs to a $\lambda$-block and $C$ otherwise
	($N$ and $C$ standing for a \textbf{n}on-cyclic and \textbf{c}yclic element, respectively). 
	We call $\vartheta(\sigma)$ the \textit{type vector} of $\sigma$.
	\begin{itemize}
		\item[$(T_1)$] \quad The type vector of any $2$-subsequence contains $CCNN$ or $NNCC$ in the first five components 
			(depending on the position of the first term in the sequence).
		\item[$(T_2)$] \quad The type vector of any $4$-subsequence equals $NCNC$ or $CNCN$ in the first four components.
	\end{itemize}

	Now, suppose the contrary, and let $\rho=\rho_1 \rho_2 = y_1\dots y_r y_{r+1}\dots y_{2r}$ be a repetition in some $d$-subsequence of $\varphi_6^t$.
	We start by analyzing possible values of $r$. 
	
	By $(R_1)$, $r \ge 3$, so suppose first that $r = 3$. We will consider the cases regarding the type vectors of $\rho$.
	By $(R_2)$ and $(R_3)$, there are six possible type vectors for $\rho$, namely:
	$CCN \ CCN$, $CNC \ CNC$, $CNN \ CNN$, $NCC \ NCC$, $NCN \ NCN$, and $NNC \ NNC$. By $(T_1)$ and $(T_2)$, such a sequence does not appear
	in any $\ell$-sequence for $\ell \in \set{2,4}$. Hence, it remains to consider $\ell \in \set{3,5,6}$.	
	Let $j$, $j \in \set{1,2,3,4}$, be the index of $y_1$ in the $\lambda$- or $\zeta$-block it belongs to.
	
	In Table~\ref{tbl:types3}, we present type vectors regarding $j$'s and $\ell$'s.
	\begin{table}[htp!]
		\centering
		\begin{tabular}{|c|c|c|c|}
			\hline
			$j \ / \ \ell$ & 3 			& 5 			& 6 		\\\hline
			$1$ 		& NNC \ NCC 	& NCN \ CCN 	& NCC \ NNC	\\\hline
			$2$ 		& NCC \ NCN 	& \textbf{NCN \ NCN} 	& NCC \ NNC	\\\hline
			$3$ 		& \textbf{NCN \ NCN} 	& NCC \ NCN 	& NNC \ CNN	\\\hline
			$4$ 		& NCN \ CCN 	& NNC \ NCC 	& NNC \ CNN	\\\hline
		\end{tabular}
		\caption{The type vectors of $\rho$ regarding $j$'s and $\ell$'s in the case $r=3$, assuming the first term $y_1$ 
			lies in a $\lambda$-block. In the symmetric case, when $y_1$ is in a $\zeta$-block, 
			the type vector values are simply interchanged.}
		\label{tbl:types3}
	\end{table}	
	The only two type vectors matching the possibilities for the type vectors of $\rho$ are in the cases $(j,\ell) \in \set{(3,3),(2,5)}$.
	In the case $(3,3)$, the indices of $y_1,\dots,y_{6}$ within their blocks are respectively $3,2,1,4,3,2$. 
	When $y_1$ belongs to a $\lambda$-block, $y_2$ and $y_5$ must belong to consecutive $\zeta$-blocks. But, then $y_2 \neq y_5$, 
	due to the construction of circular sequences.
	On the other hand, if $y_1$ belongs to a $\zeta$-block, then $y_2$ is at index $2$ in a $\lambda$-block 
	and $y_5$ is at index $3$ in a $\lambda$-block, so again $y_2 \neq y_5$, due to $(L_3)$.
	
	In the case $(2,5)$, the indices of $y_1,\dots,y_{6}$ within their blocks are respectively $2,3,4,1,2,3$. 
	Suppose first that $y_1$ belongs to a $\zeta$-block. 
	Then, $y_2$ is at index $3$ in a $\lambda$-block 
	and $y_5$ is at index $2$ in a $\lambda$-block, and hence $y_2 \neq y_5$, due to $(L_3)$.
	Finally, suppose $y_1$ belongs to a $\lambda$-block. 
	Then, $y_2$ is at index $3$ in a $\zeta$-block and $y_5$ is at index $2$ in a $\zeta$-block,
	however, the two $\zeta$-blocks are not consecutive, and hence $y_2 \neq y_5$.
	It follows that $r \ge 4$.
	
	Using the construction properties of circular sequences, 
	we can obtain additional properties of $r$ regarding the structure of type vectors.
	\begin{claim}
		\label{cl:zeta_consec}
		If there are two consecutive $\zeta$-terms in $\rho_1$, then $32$ divides $d \cdot r$.
	\end{claim}
	Note that we do not require the two terms being in the same $\zeta$-block.
	\begin{proofclaim}	
		We prove the claim by showing that having two consecutive $\zeta$-terms, $x_{i}$ and $x_{i+d}$, in $\rho_1$ imply 
		that the corresponding two $\zeta$-terms , $x_{j}$ and $x_{j+d}$, in $\rho_2$ must appear at the same indices in their $\zeta$-blocks. 
		This fact further implies that the difference between $i$ and $j$ is $(8 \cdot 4)t$ 
		($8$ since each pair of $\lambda$- and $\zeta$-blocks has $8$ terms, 
		and $4$, since $\zeta$-blocks have period $4$ in $\cir{4}{4^t}$), for some positive integer $t$. 
		On the other hand, there are $d \cdot r$ terms between $x_i$ and $x_j$, and hence $32$ divides $d \cdot r$.
		
		We consider the cases regarding $d$. For $d=1$, the claim is trivial. 
		For $d=2$, the pair of terms $x_{i}$ and $x_{i+2}$ can appear twice in four distinct $\zeta$-blocks.
		However, since the parity of the indices $i$ and $j$ must be the same in this case, they must
		appear in the same $\zeta$-block in $\rho_2$.
		
		In the case $d=3$ a pair of two symbols appear only once in four distinct $\zeta$-blocks, hence there is nothing to prove.
		In the case $d=4$, it is not possible to have two consecutive $\zeta$-terms.
		
		In the cases $d=5$ and $d=6$, the two terms belong to two consecutive $\zeta$-blocks.
		In the former, there is again only one appearance of each pair per four blocks, so it remains to consider the case $d=6$.
		There are two possible appearances of a pair, but since the indices must have the same parity, the pair must appear in the
		same two $\zeta$-blocks. This completes the proof of the claim.
	\end{proofclaim}
	
	We continue by considering the cases regarding $d$.
	\begin{itemize}
		\item{} $d=2$.
			
			If there are no two consecutive terms of $\rho_1$ that belong to the same $\zeta$-block, then $r = 4$
			and $y_1$ is a part of a $\zeta$-block.
			But in this case, there are two consecutive $\lambda$-blocks, uniquely determined by $y_2$, $y_3$ and 
			$y_6$, $y_7$, which must be equal as $y_2 = y_6$ and $y_3 = y_7$, a contradiction to Lemma~\ref{lem:k6}.
			
			So, there is at least one $\zeta$-block which contains two terms of $\rho_1$. 
			By Claim~\ref{cl:zeta_consec}, $r$ is divisible by $32/2 = 16$. 
			Suppose $y_1$ is the first (resp. the second) term of some $\lambda$-block.
			Then, $y_2$, $y_{r+1}$, and $y_{r+2}$ are also $\lambda$-terms, and hence every $\lambda$-block of $\rho$ is uniquely determined.
			By Lemma~\ref{lem:deter}, it follows there is a repetition in $\slam{t}$, a contradiction to Lemma~\ref{lem:k6}. 
			
			Now, suppose $y_1$ is the third (resp. the fourth) term of some $\lambda$-block $\gamma_1$.
			Let $\gamma_2$ be the $\lambda$-block determined by $y_r$ and $y_r+1$. Clearly, $\gamma_1 \neq \gamma_2$,
			otherwise there is a repetition in $\slam{t}$, by Lemma~\ref{lem:deter}. 
			However, since the third and the fourth terms of $\gamma_1$ and $\gamma_2$ are equal, 
			they are either mapped by $\lambda$ from $\set{1,4}$ or $\set{2,3}$. 
			As the $\lambda$-block in $\sigma_1$ following $\gamma_1$ is the same as the $\lambda$-block in $\sigma_2$ following $\gamma_2$,
			we obtain a contradiction due to $(L_6)$.
			
			Finally, suppose $y_1$ is a part of a $\zeta$-block. Since $r$ is divisible by $16$,
			the number of $\lambda$-blocks in each of $\rho_1$ and $\rho_2$ is divisible by $4$,
			and since all of them are uniquely determined, we have a repetition in $\slam{t}$ 
			(in fact already in $\slam{t-1}$), a contradiction.
			
		\item{} $d=3$.
		
			We first show that there are two consecutive $\zeta$-terms in $\rho_1$.
			Suppose the contrary. Then, since $r > 3$ and the fact that the type 
			vectors of $\rho_1$ and $\rho_2$ must match, there are two consecutive $\lambda$-terms in $\rho_1$.
			But, in the type vector, between two pairs of two consecutive $\lambda$-terms, for $d=3$, 
			there are two consecutive $\zeta$-terms, a contradiction.
			
			Hence, we may assume there are two consecutive $\zeta$-terms in $\rho_1$ and by Claim~\ref{cl:zeta_consec},
			$32$ divides $3r$. Observe also that for $d=3$, there is at least one term from $\rho$ in every 
			$\lambda$-block of the covering sequence of $\rho$. Then, by $(L_7)$, we infer that all 
			$\lambda$-blocks in the covering sequences of $\rho_1$ are equal to the corresponding 
			$\lambda$-blocks in the covering sequences of $\rho_2$, and hence there is a repetition in $\slam{t-1}$, a contradiction.
		
		\item{} $d=4$. 
		
			In this case, all the terms of $\rho$ are at the same indices in their $\lambda$- and $\zeta$-blocks.
			As there is at least one $\zeta$-term, by construction of circular sequences, 
			we have that $32$ divides $4r$, and hence $8$ divides $r$. 
			Thus, by $(L_7)$, all $\lambda$-blocks in the covering sequences of $\rho_1$ are equal to the corresponding ones in the covering sequences of $\rho_2$, 
			and hence there is a repetition in $\slam{t-1}$, a contradiction.
		
		\item{} $d=5$.
		
			In this case, $\lambda$- and $\zeta$-blocks of the covering sequence of $\rho_1$ contain 
			precisely one term from $\rho_1$ with an exception of every fifth block, which is being skipped.
			Hence, there are two consecutive $\lambda$- or $\zeta$-terms in $\rho_1$ as soon as $r > 4$. 
			As $r > 3$, the only possible $\rho_1$ with no consecutive terms of the same type has length $4$.
			However, in such a case, the terms $y_1$ and $y_{r+1}$ are not of the same type, so $r > 4$.
			
			Suppose first there are no consecutive $\zeta$-terms in $\rho_1$. In that case, there are 
			two consecutive $\lambda$-terms in $\rho_1$, and hence also in $\rho_2$. Moreover, since
			between every pair of consecutive $\lambda$-terms there are two consecutive $\zeta$-terms,
			the only possible $r$ for such $\rho$, satisfying also that the type vectors of $\rho_1$ and $\rho_2$
			are the same, is $8$. However, then $y_1$ and $y_{r+1}$ are both $\zeta$-terms but the difference
			between their indices in the covering sequence is $5\cdot 8 = 40$, meaning that $y_{1} \neq y_{r+1}$.
			
			So, we may assume there are two consecutive $\zeta$-terms in $\rho_1$ and, by Claim~\ref{cl:zeta_consec},
			$32$ divides $5r$ (hence $32$ divides $r$ also). 
			By $(L_7)$, all $\lambda$-blocks in the covering sequence of $\rho_1$ that contain one term from $\rho_1$
			are equal to the corresponding $\lambda$-blocks in the covering sequence of $\rho_2$.			
			Furthermore, since in the covering sequence of $\rho$ three out of every four $\lambda$-blocks contain 
			one term from $\rho$, also the $\lambda$-block $\gamma_0$ without a term is uniquely determined, unless it is the first
			$\lambda$-block of $\hat{\rho_1}$ or the last $\lambda$-block of $\hat{\rho_2}$. 
			In the case when $\gamma_0$ is determined, the covering sequence of $\rho_1$ contains the same sequence of $\lambda$-blocks as
			the covering sequence of $\rho_2$, and so there is a repetition in $\slam{t-1}$, a contradiction.
			
			Hence, we may assume $\gamma_0$ is not uniquely determined, and without loss of generality, suppose it is 
			the first $\lambda$-block of $\hat{\rho_1}$. 
			Since $\gamma_0$ is not uniquely determined, it is mapped from either the third or the fourth symbol of some $\lambda$-block $\xi_0$ of $\slam{t-1}$.
			In the former case, $\xi_0$ is completely determined, since $r \ge 32$ and one can determine the $\lambda$-block following $\xi_0$ in $\slam{t-1}$,
			and hence also $\gamma_0$ is completely determined.
			In the latter case, observe that, $y_2\dots y_{r+1}y_{r+2}\dots y_{2r}x_{j+5}$ (with $y_{2r} = x_j$) is also a repetition,
			and considering it, we have all $\lambda$-blocks in $\hat{\rho}$ determined, a contradiction.
			
		\item{} $d=6$.
		
			In this case, $\rho_1$ alternately contains two consecutive $\lambda$- and two consecutive $\zeta$-terms,
			with a possible shift in the beginning depending on the index of first term in the covering sequence of $\rho$.
			Hence, as $r > 3$, there are always two consecutive $\zeta$-terms in $\rho_1$ unless $r=4$ and 
			the type vector of $\rho_1$ is $CNNC$. However, in that case $y_1 \neq y_{r+1}$ by the construction of circular sequences.
			
			Thus, we may assume there are two consecutive $\zeta$-terms in $\rho_1$ and, by Claim~\ref{cl:zeta_consec}, 
			$32$ divides $6r$, hence $16$ divides $r$. 
			Let $r = 16t$; then the length of the covering sequence of $\rho_1$ is $6 \cdot 16t = 96t$ 
			and therefore there are $12t$ $\lambda$-blocks, where every two out of three consecutive $\lambda$-blocks contain a term from $\rho_1$.			
			By $(L_7)$, all $\lambda$-blocks in the covering sequence of $\rho_1$ that contain one term from $\rho_1$
			are equal to the corresponding $\lambda$-blocks in the covering sequence of $\rho_2$. 
			Recall that a $\lambda$-block is not uniquely determined by one term; it can be one of two possible (see $(L_8)$).
			
			Let $\sigma^{t}$ be the covering sequence of $\rho$ with all $\zeta$-blocks removed and let
			$\sigma^{t-1} = \lambda^{-1}(\sigma)$. Clearly, $\sigma^{t-1}$ is a subsequence of $\slam{t-1}$.
			Let $\sigma_1^{t-1}$ and $\sigma_2^{t-1}$ be the sequences defined analogously for $\rho_1$ and $\rho_2$, respectively.
			As we deduced above, $\sigma_1^{t-1} = z_1z_2\dots z_{12t}$ has $12t$ elements. 
			We consider four subcases regarding the index of $z_1$ in $\slam{t-1}$.
			Note that in each of the four cases, for every $z_i$ that is a preimage of some $\lambda$-block with one term from $\rho$,
			we can uniquely determine which symbol $z_i$ represents simply by $(L_8)$ and the position of $z_i$ in the $\lambda$-block of $\slam{t-1}$.
			Consequently, every ``complete'' $\lambda$-block of $\sigma_1^{t-1}$ is uniquely determined by $(L_3)$, since we know at least two of its terms, and in the case,
			when two terms are known, they are at indices $2$ and $3$.
			
			Suppose first $z_1$ is at index $4i+1$ in $\lambda^{t-1}$ for some $i$. 
			Then, there are $3t$ complete uniquely determined $\lambda$-blocks in $\sigma_1^{t-1}$, and hence by Lemma~\ref{lem:deter}, 
			there is a repetition also in $\slam{t-1}$, a contradiction.
			
			Next, suppose $z_1$ is at index $4i+4$ in $\lambda^{t-1}$. There are $3t-1$ complete uniquely determined $\lambda$-blocks in $\sigma_1^{t-1}$
			and one $\lambda$-block, with $3$ terms $z_{12t-2}z_{12t-1}z_{12t}$. However, as argued above, the latter is also uniquely determined, which
			means that $z_2\dots z_{24t}w$, where $w$ is the element at index $4i+4+24t+1$ in $\slam{t-1}$, is also a repetition, and hence we may use the
			argumentation for $z_1$ being at index $4i+1$.
			
			Suppose $z_1$ is at index $4i+3$ in $\lambda^{t-1}$. In this case, there are $3t-2$ complete uniquely determined $\lambda$-blocks in $\sigma_1^{t-1}$,
			and two $\lambda$-blocks having two terms in $\sigma_1^{t-1}$. The second one, $z_{12t-1}z_{12t}z_{12t+1}z_{12t+2}$ has the other two terms in $\sigma_2^{t-1}$. 
			Now, if $3t$ is divisible by $4$, then the $\lambda$-block $z_{-1}z_{0}z_1z_2$ equals $z_{12t-1}z_{12t}z_{12t+1}z_{12t+2}$ and we again can shift the 
			sequence to the left as above, obtaining a repetition.			
			Hence $3t$ is not divisible by $4$. 
			Consider the $\lambda$-blocks $\alpha_1=z_{3}z_{4}z_{5}z_{6}$ and $\alpha_1=z_{7}z_{8}z_{9}z_{10}$. They are uniquely determined and they must be 
			equal to the $\lambda$-blocks $z_{3+12t}z_{4+12t}z_{5+12t}z_{6+12t}$ and $z_{7+12t}z_{8+12t}z_{9+12t}z_{10+12t}$, which is not possible due to
			the parity condition and $(L_6)$.
			
			Finally, suppose $z_1$ is at index $4i+2$ in $\lambda^{t-1}$. Again, if $3t$ is divisible by $4$, then we shift the sequence by one to the right 
			(I will write this nicer), as the first (incomplete) $\lambda$-block in $\sigma_1^{t-1}$ must match the first (incomplete) $\lambda$-block in $\sigma_2^{t-1}$,
			and we obtain a repetition in $\slam{t-1}$. Otherwise, $3t$ is not divisible by $4$, and we obtain a contradiction on the equality of first
			two complete $\lambda$-blocks in $\sigma_1^{t-1}$ and $\sigma_2^{t-1}$.
	\end{itemize}
\end{proof}

\section{Discussion}

In this paper, we improve the current state of Conjecture~\ref{conj:k+2} by showing that it is true for every integer $k$ between $1$ and $8$.
In particular, we present two proving techniques, which are in their essence similar, but very much different in practice. 
Namely, the proving technique presented in Section~\ref{sec:pas} is (provided there are available computing resources) efficient for confirming Conjecture~\ref{conj:k+2}
for small $k$'s, since one can employ computing resources to verify small instances, while the statement of the conjecture then holds almost trivially for larger instances. 
However, to be able to prove Conjecture~\ref{conj:k+2} in general or at least for an infinite number of integers, it will fail.

On the other hand, the method described in Section~\ref{sec:cons} is more promising. 
We are using a special construction of Thue sequences with properties allowing to prove that they are also $k$-Thue. 
This technique needs more argumentation for proving that the generated sequences are indeed $k$-Thue, but allows for establishing properties 
for a bigger set of $k$'s, possibly infinite.

\bibliographystyle{plain}
\bibliography{MainBib}

%%%%%%%%%%%%%%%%%%%%%%%%%%%%%%%%%%%%%%%%%%%%%%%%%%%%%%%%%%%%%%%%%%%%%%%%%%%%%%

\end{document}